\documentclass[11pt,a4paper]{article}
\usepackage[margin=3cm]{geometry}
\usepackage[dvips]{graphicx}
\usepackage{subfig}
\usepackage[affil-sl]{authblk}
\usepackage{dsfont}
\usepackage{amsmath}
\usepackage{amsthm}
\usepackage{amssymb}
\usepackage{mathtools}
\usepackage{verbatim}
\usepackage{algorithmic}
\usepackage{algorithm}
\usepackage[inline]{enumitem}
\usepackage{framed}
\usepackage{marginnote}
\usepackage[dvipsnames]{xcolor}
\usepackage{nicefrac}
\usepackage{hyperref}
\hypersetup{
    colorlinks=true,
    linkcolor=blue,
    filecolor=magenta,      
    urlcolor=cyan,
    citecolor=[RGB]{0,200,0}
}
\usepackage[square,numbers]{natbib}

\usepackage{tikz}
\usepackage{calc}
\usepackage[textwidth=3cm]{todonotes} 

\usetikzlibrary{decorations.markings,calc, shapes, arrows.meta, positioning,shapes.multipart}
\tikzstyle{every picture}=[line width=.65pt,minimum size=3pt,every label/.append style={font=\small},label distance=-2pt]
\tikzstyle{every node}=[font=\small]
\tikzset{>=stealth}
\tikzstyle{vtx}=[circle,draw,thick,fill=black!50]
\tikzstyle{vtxr}=[circle,draw,thick,fill=red!50]
\tikzstyle{vtxp}=[circle,draw,thick,fill=pink!50]
\tikzstyle{vtxb}=[circle,draw,thick,fill=blue!30]
\tikzstyle{vtxgre}=[circle,draw,thick,fill=green!60]

\newcommand{\R}{\mathds{R}}
\newcommand{\Q}{\mathds{Q}}
\newcommand{\Z}{\mathds{Z}}
\newcommand{\N}{\mathds{N}}
\newcommand{\floor}[1]{\left\lfloor#1\right\rfloor}

\newcommand{\conv}{\operatorname{conv}}

\newcommand{\rec}{\operatorname{rec}}

\renewcommand{\vert}{\operatorname{vert}}
\renewcommand{\epsilon}{\varepsilon}

\newcommand{\rc}{\operatorname{rc}}

\renewcommand{\int}{\operatorname{int}}
\newcommand{\lift}{\operatorname{lift}}
\newcommand{\clift}{\operatorname{clift}}

\newcommand{\cF}{\mathcal{F}}
\newcommand{\cX}{\mathcal{X}}
\newcommand{\cY}{\mathcal{Y}}

\newcommand{\cT}{\mathcal{T}}

\newcommand{\setcond}[2]{\left\{ #1 \, :\, #2 \right\}} 
\newcommand{\tih}{{\tilde{h}}}

\newcommand{\card}[1]{| #1 |}

\newcommand{\sucn}{\operatorname{sucn}}
\newcommand{\slcn}{\operatorname{slcn}}

\newcommand{\eps}{\varepsilon}

\newcommand{\define}{\coloneqq}

\newtheorem{prop}{Proposition}
\newtheorem{obs}[prop]{Observation}
\newtheorem{remark}[prop]{Remark}
\newtheorem{lemma}[prop]{Lemma}

\newtheorem{cor}[prop]{Corollary}
\newtheorem{theorem}[prop]{Theorem}

\theoremstyle{definition}
\newtheorem*{question}{Questions}

\begin{document}

\title{The role of rationality in integer-programming relaxations}
\author[1]{Manuel Aprile}
\author[2]{Gennadiy Averkov}
\author[1]{Marco Di Summa}
\author[3]{Christopher Hojny}
\affil[1]{%
Dipartimento di Matematica ``Tullio Levi-Civita'',
Universit\`a degli Studi di Padova, 
Via Trieste 63,
35121 Padova, Italy, emails disumma@math.unipd.it, manuelf.aprile@gmail.com
}
\affil[2]{%
  BTU Cottbus-Senftenberg\\
  Platz der Deutschen Einheit 1\\
  03046 Cottbus, Germany\\
  \emph{email} averkov@b-tu.de
}
\affil[3]{%
  Eindhoven University of Technology\\
  Combinatorial Optimization Group\\
  PO Box~513\\
  5600 MB Eindhoven, The Netherlands\\
  \emph{email} c.hojny@tue.nl
}

\maketitle

\begin{abstract}

  For a finite set~$X \subset \Z^d$ that can be represented as $X = Q \cap
  \Z^d$ for some polyhedron $Q$, we call $Q$ a relaxation of $X$ and define 
  the relaxation complexity~$\rc(X)$ of~$X$ as the least number of facets
  among all possible  relaxations~$Q$ of~$X$. The rational relaxation
  complexity~$\rc_\Q(X)$ restricts the definition of~$\rc(X)$ to rational
  polyhedra $Q$.
  In this article, we focus on~$X = \Delta_d$, the vertex set of the standard simplex, 
   which consists of the null
  vector and the standard unit vectors in~$\R^d$.
  We show that~$\rc(\Delta_d) \leq d$ for every~$d
  \geq 5$.
  That is, since $\rc_\Q(\Delta_d)=d+1$,
   irrationality can reduce the minimal size of relaxations.
  This answers an open question posed by Kaibel and Weltge (Lower bounds on
  the size of integer programs without additional variables,
  \emph{Mathematical Programming}, 154(1):407--425, 2015).
  Moreover, we prove the asymptotic statement~$\rc(\Delta_d) \in
  O(\frac{d}{\sqrt{\log(d)}})$, which shows that the ratio $\rc(\Delta_d)/\rc_\Q(\Delta_d)$ goes to~$0$, as $d\to \infty$. 
\end{abstract}

\section{Introduction}
Finding compact representations of a large set of integer points in a high-dimensional space is a recurring theme in combinatorial optimization, where these points usually describe the feasible solutions to a discrete problem. A classical approach is, given a set of points of interest $X\subseteq \Z^d$, to find a polyhedron $Q\subseteq \R^d$ whose integer points are exactly those in $X$, i.e., such that $Q\cap \Z^d = X$. Such a polyhedron is called a \emph{relaxation} of $X$. By feeding a  relaxation of $X$ to an integer programming solver, one can optimize a linear objective function over $X$. 
The \emph{relaxation complexity} of $X$, denoted by $\rc(X)$, is the smallest number of facets of a relaxation of $X$. This concept was introduced in \cite{kaibel2015lower} (see also \cite{weltge2015diss}) as a measure of the complexity of a set of integer points.

A basic question is to determine the relaxation complexity of simple sets
like the vertices of a hypercube or a simplex.
For instance, one has $\rc(\{0,1\}^d)\leq 2d$ by counting the number of
facets of the 0/1 hypercube. However, a relaxation with only $d+1$ facets
exists~\cite{kaibel2015lower}.
Moreover, it is easy to see that any \emph{bounded} relaxation of a full-dimensional set
$X\subseteq \Z^d$ must have at least $d+1$ facets, and one can further show
that the 0/1 hypercube does not admit any unbounded relaxation
\cite{kaibel2015lower}, implying $\rc(\{0,1\}^d)= d+1$.
In general, however, a finite set of points can admit an unbounded
relaxation. Notice that such a relaxation must have irrational rays only (assuming $X\ne\emptyset$):
indeed, starting from a point of $X$ and applying a rational ray would
produce infinitely many integer points.
As a prominent example, the authors of \cite{kaibel2015lower} describe a 5-dimensional
simplex (that is unimodularly equivalent to the standard simplex) and an
unbounded, irrational relaxation of its vertex set, which however has at least six facets.
Inspired by this result, they pose the following questions.
\begin{question}
  Let~$\Delta_d$ be the set consisting of the null vector and the
  canonical unit vectors in~$\R^d$, and let~$\rc_\Q(\cdot)$
  be the restriction of~$\rc(\cdot)$ to rational relaxations.
  \begin{enumerate}[label=(Q\arabic*)]
  \item\label{Q1} Does~$\rc(\Delta_d) = \rc_\Q(\Delta_d) = d+1$ hold for
    all positive integers~$d$?
  \item\label{Q2} If~$X \subseteq \Z^d$ is finite, does~$\rc(X) \geq
    \dim(X) + 1$ hold?
    Here, $\dim(X)$ is the dimension of the affine hull of~$X$.
  \item\label{Q3} If~$X \subseteq \Z^d$ is finite, does~$\rc(X) =
    \rc_\Q(X)$ hold?
  \end{enumerate}
\end{question}

In this paper, we show that $\rc(\Delta_d)<\rc_\Q(\Delta_d)$ for $d\geq 5$.
This gives negative answers to all the three questions above.
The starting point is an example in
dimension~5.
 
\begin{theorem}\label{thm:dim5}
  We have $\rc(\Delta_5)=5$.
\end{theorem}

From this example, we can easily show that irrationality helps to reduce
the relaxation complexity also in higher dimensions.
Thus, Questions~\ref{Q1}--\ref{Q3} have a negative answer for every~$d \geq
5$.

\begin{cor}\label{cor:dimd}
	For every $d \in \Z_{>0}$, one has 
	\[
\rc(\Delta_d) \le 5 \floor{\frac{d+1}{6}} + \Bigl((d+1) \bmod 6\Bigr),
\]
where $\floor{\,.\,}$ denotes the rounding-down operation and $(d+1) \bmod 6$ is the remainder of the division of $(d+1)$ by $6$. 
\end{cor}

Since the relaxation complexity and rational relaxation complexity differ,
one might wonder about the ratio $\rc(\Delta_d)/ \rc_\Q(\Delta_d)$ of these two
quantities. By Corollary~\ref{cor:dimd}, this ratio is bounded from above by $\frac{5}{6} + o(d),$ as $d \to \infty$. 
With a significantly more involved construction, we can show that $\rc(\Delta_d)/ \rc_\Q(\Delta_d) \to 0,$ as $d \to \infty$, by providing the following asymptotic estimate:

\begin{theorem}\label{thm:main}
We have $\rc(\Delta_d)\in O\left(\frac{d}{\sqrt{\log d}}\right)$.
\end{theorem}

The remainder of this article is structured as follows.
First, we mention related literature (Section~\ref{sec:literature}) and
introduce the notation and terminology that we use throughout the article
(Section~\ref{sec:notation}).
In Section~\ref{sec:thm1}, we prove Theorem~\ref{thm:dim5} and
Corollary~\ref{cor:dimd}, and, based on the proof of Theorem~\ref{thm:dim5},
we derive our strategy for proving Theorem~\ref{thm:main}.
The proof of the latter theorem is then presented in Section~\ref{sec:thm3}.
Finally, we list some open problems in Section~\ref{outlook}.

\subsection{Related Literature}
\label{sec:literature}

The relaxation complexity has been formally introduced
in~\cite{kaibel2015lower}, where it is also shown how to derive lower
bounds for it.
Among others, this lower bound shows that any relaxation of a natural
encoding of Hamiltonian cycles via binary vectors needs exponentially many
inequalities in general.
For~$X \subseteq \{0,1\}^d$, an upper bound on the minimal number of
inequalities needed to separate~$X$ and~$\{0,1\}^d \setminus X$ has been
derived in~\cite{jeroslow1975ondefining}.
Complementing the upper and lower bounds, \cite{averkov2021complexity}
derived criteria that guarantee that~$\rc(X)$ can be computed by solving a
mixed-integer program.
More sophisticated mixed-integer programming techniques to compute~$\rc(X)$
have been discussed recently in~\cite{AverkovEtAl2022}.
Independently from the criteria by~\cite{averkov2021complexity} that
guarantee computability in arbitrary dimensions, \cite{weltge2015diss} shows
that~$\rc(X)$ for~$X \subseteq \Z^d$ is
computable if~$d = 2$, and \cite{averkov2021complexity} extended this
result to $d = 3$.
In  \cite{AverkovEtAl2021} it was shown that, if~$d = 2$, then~$\rc(X)$ can
be found in polynomial time. Furthermore, \cite{AverkovEtAl2021} also addresses issues behind the computability of $\rc(X)$ and $\rc_\Q(X)$ for an arbitrary dimension and reveals that having $\rc(X)<\rc_\Q(X)$ is an additional obstacle on the way to computing $\rc(X)$ and $\rc_\Q(X)$, as certain natural computational procedures fail to stop and so do not give finite algorithms in this case. For more details, see Section~\ref{outlook}.

Extension complexity is a related, but more well studied concept than relaxation complexity (see \cite{conforti2013extended} for a survey). The \emph{extension complexity} of a polyhedron $P\subseteq \R^d$ is the minimum number of facets of a polyhedron $Q\subseteq \R^{d+d'}$ that projects down to $P$. If $P$ is the convex hull of our set of integer points $X$, the extension complexity of $P$ measures how compactly one can describe $X$ using a linear formulation with additional variables, while the relaxation complexity of $X$ allows to use integer formulations, but with no additional variables. The existence of small extended formulations for a given polyhedron~$P$ is related to the existence of small non-negative factorizations of a matrix related to $P$ by the celebrated Yannakakis' Theorem \cite{yannakakis1991expressing}. Finding non-negative factorizations of (non-negative) matrices is a problem of independent interest, and in 1993 it was asked \cite{cohen1993nonnegative} whether using irrational numbers could lead to smaller factorizations of rational matrices. This was recently answered in \cite{chistikov2017nonnegative} and \cite{shitov2016nonnegative}: irrational numbers indeed help, as there are rational matrices whose minimum size factorizations require irrational numbers. As far as the authors know, the examples given in the aforementioned papers do not correspond to polyhedra, hence it is open whether using irrational numbers leads to smaller extended formulations for rational polyhedra.

Despite the lower bounds and the computability results mentioned above, the existence of small relaxations for a given set of points does not have a general characterization similar to what is known for  extended formulations in terms of non-negative factorizations and communication protocols \cite{faenza2015extended, yannakakis1991expressing}. Moreover, we lack general approaches to construct relaxations, while several efficient constructions for extended formulations are known \cite{aprile2022extended, aprile2020extended,  kaibel2010branched, tiwary2020extension}. This paper can be seen as one step forward in this direction.

Interestingly, sets of points like those corresponding to the spanning trees of a graph do not admit relaxations of subexponential size \cite{kaibel2015lower}, but do admit polynomial-sized extended formulations \cite{aprile2021smaller,martin1991using,wong1980integer}; the opposite happens for other sets, for instance those corresponding to stable sets or matchings in a graph, which admit simple, linear-sized relaxations, but no extended formulation of subexponential size \cite{fiorini2015exponential, rothvoss2017matching}.

\subsection{Notation and Terminology}\label{sec:notation} 

We use $\subseteq$ and $\subset$ to denote inclusion and strict inclusion, respectively. 
As usual, we define $[m] \define \{1,\ldots,m\}$ for $m \in \Z_{>0}$ and $[m] \define \emptyset$ for $m =0$. 
We denote by $\mathbf{0}_a$ (resp.~$\mathbf{0}_{a, b}$) the all-zero vector
(resp.\ all-zero matrix) of  size $a$ (resp.~$a \times b$), writing only $\boldsymbol{0}$
when the dimension is clear from the context.
We denote the \emph{standard unit vectors} in $\R^d$ by $e_1,\dots,e_d$.
Using this notation, $\Delta_d = \{\mathbf{0},e_1,\dots,e_d\}$.
The \emph{dimension}~$\dim(X)$ of $X \subseteq \R^d$ is the dimension of its
affine hull.
The \emph{cardinality} of~$X$ is denoted by~$|X|$ and its \emph{convex
  hull} is denoted by~$\conv(X)$.
For two sets~$X$ and~$Y$, we denote by~$Y^X$ the \emph{set of all maps} from $X$
to $Y$.
If~$X$ is finite, then $\R^X$ is a $|X|$-dimensional vector space over
$\R$ and, thus, it has the standard Euclidean topology.
As usual, for $\phi\colon  X \to Z$ and $Y \subseteq X$, the notation~$\phi|_Y$
is used to denote the \emph{restriction} of the map $\phi$ to $Y$.

We call elements of $\Z^d$ \emph{integer points/vectors} and we call elements of~\mbox{$\Z^d \times \R^m$} \emph{mixed-integer points/vectors}. 
An affine transformation $\phi\colon \R^d \to \R^d$ is said to be \emph{unimodular} if it  maps the integer lattice onto itself, that is, $\phi(\Z^d) = \Z^d$. It is known and easy to check (see, e.g., \cite{Schrijver1978}) that $\phi$ is unimodular if and only if it can be represented as $\phi(x) = U x + v$ where~$v \in \Z^d$ and $U \in \Z^{d \times d}$ is a unimodular matrix, which is an integer matrix with determinant equal to one in absolute value. 

The \emph{recession cone} $\rec(P)$ of a nonempty polyhedron $P \subseteq \R^d$ is the set of all $u \in \R^d$ satisfying $p+u \in P$ for each $p \in P$. The recession cone of a polyhedron $P\ne\emptyset$ given by an inequality description $A x \le b$ is a  polyhedral cone, which can be given by the inequality description $A x \le \mathbf{0}$. 

For $X \subseteq Y \subseteq \R^k$, a polyhedron $P$ satisfying $X = P \cap Y$ is called a (polyhedral) \emph{relaxation} of $X$ within $Y$. 
By $\rc(X,Y)$ we denote the minimal $k$ such that $X = P \cap Y$ holds for some polyhedron $P$ with $k$ facets. If such $k$ does not exist, we let $\rc(X,Y ) =\infty$. We call $\rc(X,Y)$ the \emph{relaxation complexity} of $X$ within $Y$. It is clear that $\rc(X,Y)$ is monotonically non-decreasing with respect to $Y$ in the sense that $X \subseteq Y \subseteq Z$ implies $\rc(X,Y) \le \rc(X,Z)$. 
For $X \subseteq \Z^d$, we  call $\rc(X)\define \rc(X,\Z^d)$  the (integer)
\emph{relaxation complexity} while,  for $X \subseteq \Z^d \times \R^n$, we
call $\rc(X,\Z^d \times \R^n$) the $(d,n)$-\emph{mixed-integer relaxation
  complexity} or the \emph{mixed-integer relaxation complexity} with
respect to $d$ integer and~$n$ real variables.

We also provide some necessary background from the theory of irrational numbers. The field $\R$ is a vector space over $\Q$ and it can be seen in different ways that $\R$ is infinite-dimensional over $\Q$. One easy non-constructive way to see this is to observe that~$\R$ is uncountable, whereas $\Q$ and by this also every finite-dimensional vector space over~$\Q$ is countable. This observation shows that for every $k \in \Z_{>0}$ there exists numbers $a_1,\ldots,a_k \in \R$ that are linearly independent over $\Q$. As a consequence, we also see that for every $k \in \Z_{>0}$ there exist numbers $b_1,\ldots,b_k \in \R$ such that $1,b_1,\ldots,b_k$ are linearly independent over $\Q$. To show the latter, one can start with numbers $a_0,\ldots,a_k \in \R$ linearly independent over $\Q$ and fix $b_i \define a_i/a_0$ for $i\in[k]_0$. Alternatively, the sequence $b_0=1,b_1,\ldots,b_k$ of numbers linearly independent over $\Q$ can also be obtained by choosing an arbitrary algebraic number $c \in \R$ of algebraic degree $k+1$ and fixing $b_i = c^i$ for $i\in[k]_0$. See~\cite[Sect.~4]{gruber1987geometry} for more details.

\section{The Projection Approach}
\label{sec:thm1}

In this section, we prove Theorem~\ref{thm:dim5} and
Corollary~\ref{cor:dimd}.
Then, we will discuss how our proof strategy can be generalized to show
Theorem~\ref{thm:main}.

\subsection{The Counterexample in Dimension Five}\label{sec:counterex}
Define 
$$\Delta\define \{\boldsymbol{0}, e_1,e_2, e_3, (1,0,1,1,0), (0,1,1,0,1)\}\subset \Z^5.$$ 
$\Delta$ is the vertex set of a 5-dimensional simplex and is unimodularly equivalent to $\Delta_5$. We work with $\Delta$ instead of $\Delta_5$ because this simplifies the calculations: giving a relaxation of~$\Delta$ will imply the existence of a relaxation of $\Delta_5$ of the same size.

Let $\ell\subset \R^5$ be the line spanned by the vector
$(0,0,0,1,\sqrt{2})$. In \cite{kaibel2015lower} it is shown that
$\conv(\Delta)+\ell$ is an (unbounded) relaxation of $\Delta$.
However, such a relaxation has more than five facets.
In order to give our relaxation, it is easier to reduce ourselves to a
4-dimensional space by projecting on a suitable hyperplane along
line $\ell$, and looking for a relaxation of the projection of $\Delta$
within the projection of $\Z^5$.
In particular, we use the following observation:

\begin{obs}\label{obs:proj}
Let $\pi\colon \R^5 \to \R^4$ be the projection 
\[
	\pi(x_1,x_2,x_3,x_4,x_5) \define (x_1,x_2,x_3, x_4 - \frac{1}{\sqrt{2}} x_5)
\] and $P \subseteq \R^4$ be a polyhedron satisfying $P\cap (\Z^3 \times \R) = \pi(\Delta)$. Then $Q\define(P\times \{0\}) + \ell$ is a relaxation of $\Delta$. In other words, one has $\rc(\Delta,\Z^5)\leq \rc(\pi(\Delta),\Z^3\times \R)$.

\end{obs}

\begin{proof}
By the choice of $\pi$,  $\pi(\Z^5)$ is a subset of $\Z^3 \times \R$. Moreover, $\ell$ is the kernel of $\pi$ and $Q$ is the pre-image  of $P$ under $\pi$.  

We show that $Q\cap\Z^5=\Delta$.
If $q\in Q\cap \Z^5$, then  $\pi(q) \in P\cap (\Z^3 \times \R) = \pi(\Delta)$. The condition $\pi(q) \in \pi(\Delta)$ implies $q\in \Delta + \ell$.  But since $\ell$ is an irrational line, its only integer point is $\mathbf{0}$. Hence the points of $\Delta$ are the only integer points in $\Delta+\ell$. As $q \in \Z^5$, we conclude that $q \in \Delta$.  Conversely, if  $q \in \Delta$, then $\pi(q) \in \pi(\Delta) \in P$. Thus, $q$ is in the pre-image of $P$ under $\pi$, that is, $q \in Q$. This shows $Q \cap \Z^5 = \Delta$ and yields the assertion. 
\end{proof}

Based on Observation~\ref{obs:proj}, for deriving $\rc(\Delta,\Z^5) \le 5$ it is sufficient to find a $(3,1)$-mixed-integer relaxation
of~$\pi(\Delta)$ with five facets.
To this end, let~$\eps \in (0,1)$ and consider the inequalities
\begin{align}
    x_1&\geq x_4, \label{constr:1}\\
    x_3&\geq x_4, \label{constr:2}\\
    \eps x_1 + x_2 + \frac{1-\eps}{1+\sqrt{2}} x_3 + \frac{(1-\eps)\sqrt{2}}{1+\sqrt{2}} x_4 &\leq 1, \label{constr:3}\\
    x_3 + \sqrt{2}x_4  &\geq 0, \label{constr:4}\\
    x_1 - (1+\eps) x_2 + x_3 - x_4 &\leq 1. \label{constr:5}
\end{align}
We use these inequalities to define the polyhedron~$P_\eps \define \{ x \in
\R^4 : x \text{ satisfies~\eqref{constr:1}--\eqref{constr:5}}\}$.
Moreover, let 
\[
	\Delta' \define \pi(\Delta)=\{\boldsymbol{0}, e_1,e_2, e_3, (1,0,1,1), (0,1,1,-\frac{1}{\sqrt{2}})\}.
\]

\newcommand\prism{
	\draw[very thick] (a0) -- (a1) ;
	\draw[very thick] (a1) -- (a2) ;
	\draw[very thick] (a2) -- (b2) ;
	\draw[very thick] (b2) -- (b1);
	\draw[very thick] (b1) -- (b0) ;
	\draw[very thick] (b0)-- (a0);
	\draw[very thick] (a1) -- (b1);
	\draw[very thick] (b0) -- (b2);
	\draw[dotted,very thick] (a2) -- (a0);
} 
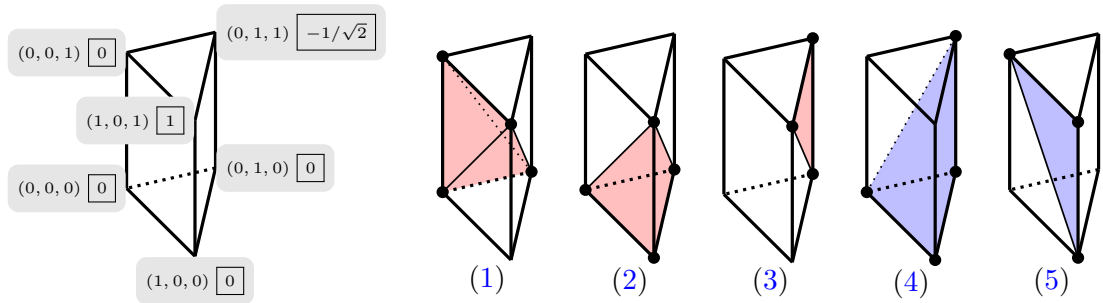
\begin{figure} 
\begin{center}
\begin{tikzpicture}[scale=0.9,baseline=2mm] 
	\tikzstyle{point} = [left,fill=black!10!white,rounded corners]
	\coordinate (a0) at (0,0) ;
	\coordinate (a1) at (1,-1) ;
	\coordinate (a2) at (1.3,0.3) ;
	\coordinate (b0) at (0,2) ;
	\coordinate (b1) at (1,1);
	\coordinate (b2) at (1.3,2.3);
	\prism 
	\node[point] at (a0) {\tiny $(0,0,0) \ \fbox{0}$ };
	\node[point,below] at (a1) {\tiny $(1,0,0) \  \fbox{0}$ };
	\node[point,right] at (a2) {\tiny $(0,1,0) \ \fbox{0}$ };
	\node[point,left] at (b0) {\tiny $(0,0,1) \ \fbox{0}$ }; 
	\node[point,left] at (b1) {\tiny  $(1,0,1) \ \fbox{1}$ }; 
	\node[point,right] at (b2) {\tiny  $(0,1,1)$  \fbox{$-1/\sqrt{2}$} }; 
\end{tikzpicture} 
\hspace{2mm}
\begin{tabular}{c}
\begin{tikzpicture}[baseline=7mm]
	\fill[fill=red!50!white,opacity=0.5] (a0) -- (a2) -- (b1) -- (b0) -- (a0) ;
	\draw (a0) -- (b1) -- (a2) ;
	\draw[dotted] (a2) -- (b0);
	\prism 
	\fill (a0) circle (0.08);
\fill (a2) circle (0.08);
\fill (b1) circle (0.08);
\fill (b0) circle (0.08);
\end{tikzpicture} 
\\ 
\eqref{constr:1}
\end{tabular} 
\begin{tabular}{c}
\begin{tikzpicture}[baseline=7mm]
	\fill[fill=red!50!white,opacity=0.5] (a0) -- (a1) -- (a2) -- (b1) -- (a0) ;
	\draw (a0) -- (b1) -- (a2) ;
	\prism 
	\fill (a0) circle (0.08);
	\fill (a1) circle (0.08);
	\fill (a2) circle (0.08);
	\fill (b1) circle (0.08);
\end{tikzpicture} 
\\
\eqref{constr:2}
\end{tabular}
\begin{tabular}{c} 
\begin{tikzpicture}[baseline=7mm]
	\fill[fill=red!50!white,opacity=0.5] (a2) -- (b1) -- (b2)  -- (a2);
	\draw (a2) -- (b1);
	\prism 
	\fill (a2) circle (0.08);
	\fill (b1) circle (0.08);
	\fill (b2) circle (0.08);
\end{tikzpicture} 
\\
\eqref{constr:3}
\end{tabular}
\begin{tabular}{c}
\begin{tikzpicture}[baseline=7mm]
	\fill[fill=blue!50!white,opacity=0.5] (a0) -- (a1) -- (a2)  -- (b2) -- (a0);
	\draw[dotted] (a0) -- (b2);
	\prism 
	\fill (a2) circle (0.08);
	\fill (a1) circle (0.08);
	\fill (b2) circle (0.08);
	\fill (a0) circle (0.08);
\end{tikzpicture} 
\\
\eqref{constr:4}
\end{tabular}
\begin{tabular}{c}
	\begin{tikzpicture}[baseline=7mm]
		\fill[fill=blue!50!white,opacity=0.5] (a1) -- (b0) -- (b1) -- (a1) ;
		\draw (a1) -- (b0);
		\prism 
		\fill (a1) circle (0.08);
		\fill (b0) circle (0.08);
		\fill (b1) circle (0.08);
	\end{tikzpicture} 
	\\
	\eqref{constr:5}
\end{tabular}
\end{center} 
\caption{The projection of  $\Delta'$ onto the first three components, with the fourth component of the elements of $\Delta'$ indicated in the boxes, (left) and the sets of elements of $\Delta'$ on which inequalities \eqref{constr:1}--\eqref{constr:5} are active (right). The colors indicate that \eqref{constr:1}--\eqref{constr:3} provide upper bounds on $x_4$, while \eqref{constr:4}--\eqref{constr:5} provide lower bounds on $x_4$. Inequalities \eqref{constr:1}, \eqref{constr:2} and~\eqref{constr:4} are facet-defining for $\conv(\Delta')$, as each of them is active on four affinely independent vertices. Inequalities \eqref{constr:3} and \eqref{constr:5} define two-dimensional faces of $\conv(\Delta')$.\label{fig:geometry:system}} 
\end{figure} 

\begin{figure} 
\begin{center}
\includegraphics[width=3cm]{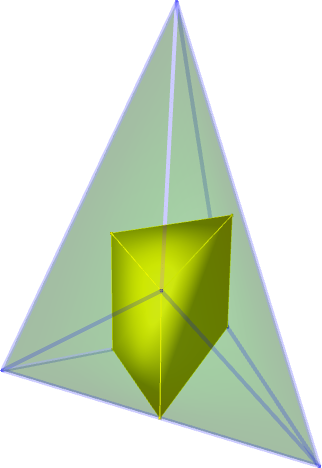}
\end{center}
\caption{The projection on the first three components of $\Delta'$ (convex hull in yellow) and the polytope~$P_{\nicefrac{1}{8}}$ (green). The  inequality description of the projection of $P_{\nicefrac{1}{8}}$ can be derived from the inequality description \eqref{constr:1}--\eqref{constr:5} of $P_{\nicefrac{1}{8}}$ using Fourier-Motzkin elimination, which  combines the three upper bounds \eqref{constr:1}--\eqref{constr:3} on $x_4$ with the two lower bounds \eqref{constr:4}--\eqref{constr:5} on $x_4$. As a result, one obtains $3 \times 2 =6$ inequalities for $x_1,x_2,x_3$, and it turns out that each of these inequalities defines a facet of the projected polytope.\label{fig:fourier:motzkin}}  
\end{figure} 

\begin{prop}\label{prop:P} The polyhedron
  $P_{\nicefrac{1}{8}}$ satisfies $P_{\nicefrac{1}{8}} \cap (\Z^3 \times \R) = \Delta'$.
\end{prop}

\begin{proof}
    Consider the coordinate projection map $\phi(x_1,x_2,x_3,x_4) \define
    (x_1,x_2,x_3)$. The inequality description of the polytope $Q\define\pi
    (P_{\nicefrac{1}{8}})$ can be determined from the inequality
    description of  $P_{\nicefrac{1}{8}}$ via Fourier-Motzkin elimination, see, e.g., the book~\cite{Schrijver1978}. Since~$Q$ is a polytope (and so a bounded set), it is straightforward to check that  $Q \cap \Z^3 = \{(0,0),(1,0),(0,1)\} \times \{0,1\} =  \phi(\Delta')$. Once this verification is done, by plugging in an arbitrary $(x_1,x_2,x_3) \in Q \cap \Z^3$ into the system  \eqref{constr:1}--\eqref{constr:5} one obtains a system of linear constraints for $x_4$. These constraints provide lower and upper bounds on $x_4$. It turns out that for each choice of  $(x_1,x_2,x_3) \in Q \cap \Z^3$, the respective best lower and upper bound on $x_4$ match and thus give exactly one choice of  $x_4$ corresponding to $(x_1,x_2,x_3,x_4) \in \Delta'$. That is, for $(x_1,x_2,x_3) \in \{(0,0,0),e_1,e_2,e_3\}$, one obtains inequalities $x_4 \le 0$ and $x_4 \ge 0$, for $(x_1,x_2,x_3) = (1,0,1)$, one obtains $x_4 \le 1$ and $x_4 \ge 1$, while for $(x_1,x_2,x_3) = (0,1,1)$, one obtains $x_4 \le - \frac{1}{\sqrt{2}}$ and $x_4 \ge - \frac{1}{\sqrt{2}}$. This shows the asserted equality $P_{\nicefrac{1}{8}} \cap (\Z^3 \times \R) = \Delta'$.
  
The computational details of this proof are comprised of the Fourier-Motzkin elimination, the verification of the equality $Q \cap \Z^3 = \{(0,0),(1,0),(0,1) \} \times \{0,1\}$ and the determination of  $\setcond{ x_4 \in \R}{ (x_1,x_2,x_3,x_4) \in  P_{\nicefrac{1}{8}}}$ for each choice of $(x_1,x_2,x_3) \in Q \cap \Z^3$. All these computations can be carried out in a straightforward manner and are therefore omitted. See  also  Figs.~\ref{fig:geometry:system} and \ref{fig:fourier:motzkin} that highlight the geometry behind these computations. 
\end{proof}

We observe that the sequence of the values $\rc(\Delta_d)$ is monotonically non-decreasing. 
\begin{prop} \label{prop:mon:rc:delta}
	$\rc(\Delta_d) \le \rc(\Delta_{d+1})$. 
\end{prop} 
\begin{proof}
	A relaxation of $\Delta_{d+1}$ with $k$ inequalities restricted to the coordinate subspace $\R^d \times \{0\}$ gives a relaxation of  $\Delta_d$ with at most $k$ inequalities. This shows the desired assertion. 
\end{proof}

Our use of  Proposition~\ref{prop:mon:rc:delta} is two-fold. We need it as an ingredient in the proof of Theorem~\ref{thm:dim5}, on the one hand. 
On the other hand, since $\rc(\Delta_d)$ grows with respect to $d$, we are justified to study the asymptotic growth of this quantity. 

We are ready to prove Theorem~\ref{thm:dim5},

\begin{proof}[Proof of Theorem~\ref{thm:dim5}]
   By Proposition~\ref{prop:P} and Observation \ref{obs:proj},
  $(P_{\nicefrac{1}{8}}\times\{0\})+\ell$ is a relaxation of $\Delta$.
  Since~$P_{\nicefrac{1}{8}}$ is defined by five inequalities, we have $\rc(\Delta) \leq 5$. On the other hand, by Proposition~\ref{prop:mon:rc:delta} we have  $\rc(\Delta)  = \rc(\Delta_5) \ge \rc(\Delta_4)  = 5$, where $\rc(\Delta_4)=5$ has been verified in  \cite{averkov2021complexity}.
\end{proof}

\begin{remark}[Computer-assisted studies] 
	The system \eqref{constr:1}-\eqref{constr:5} has been discovered and validated using a computer.  For the separation of  $\Delta'$ from $\Z^4$ with five inequalities, a mixed-integer programming model has been established, akin to the model suggested in  \cite{averkov2021complexity}.  
	In the established model,  instead of $\Z^4$ a finite sufficiently large subset of $\Z^4$ has been used as an ``approximation'' of $\Z^4$. The model was represented and solved in the floating-point arithmetic. The values that have been produced as a result have been inspected and manually converted to exact algebraic numbers in the field $\Q[\sqrt{2}]\define \setcond{ a +  b\sqrt{2} }{a,b \in \Q}$. 
	
	As a post-processing step, for the validation of the system, along with the manual verification, we also used the computer algebra system SageMath \cite{sagemath}. Notably, SageMath can carry out exact polyhedral computations over the field of algebraic numbers. So, in SageMath it is possible to carry out all of the computational steps described in the above proof by working with polyhedra over the algebraic numbers. 
	In particular,  we used SageMath to determine an appropriate choice of $\epsilon>0$ in \eqref{constr:1}-\eqref{constr:5}. 
\end{remark} 

\subsection{Extensions of the Five-Dimensional Counterexample}

To prove Corollary~\ref{cor:dimd}, we show that $\rc(\Delta_d)$ is sub-additive in the cardinality $d+1$  of $\Delta_d$. To this end, we introduce the free join of non-empty finite sets $X \subset\Z^k$ and $Y \subset\Z^\ell$ and study its properties. For $k, \ell \in \Z_{>0}$, we define the \emph{free join} by 
\[
	X \ast Y \define \setcond{(x,\mathbf{0},0)}{x \in X} \cup \setcond{(\mathbf{0},y,1)}{y \in Y} \subset \Z^k \times \Z^\ell \times \Z.
\]
This operation corresponds to the free join of polytopes, as $\conv(X \ast Y)$ is the free join of the polytopes $\conv(X)$ and $\conv(Y)$, as defined in \cite[15.1.3]{henk200416}. It is also useful  to define the free join in the case when $k$ or $\ell$ is zero, which means that one takes the free join with $\Delta_0 \define \{\mathbf{0}\} \subseteq \R^0$. This is naturally defined as $X \ast \Delta_0 \define (X \times \{0\}) \cup \{e_{k+1}\} \subset \Z^k \times \Z$ and $\Delta_0 \ast Y \define \{\mathbf{0}\} \cup (Y \times \{1\}) \subset \Z \times \Z^\ell$. The analogy for  passing from $X$ to $X \ast \Delta_0$ within the theory of polyhedra is building a pyramid over a given polytope. 
Note that $(X+u) \ast (Y+v)$ coincides with $X \ast Y$ up to a unimodular transformation for all $u \in \Z^k$ and $v \in \Z^\ell$. Indeed, $(X+u) \ast (Y+v)$ is the image of $X \ast Y$ under the affine unimodular map $\phi(x,y,z) \define (x + (1-z) u ,y+ z v,z)$. 

In order to handle special cases uniformly, it will we also be convenient to fix $X \ast \emptyset \define X$, $\emptyset \ast Y \define Y$, $\rc(\Delta_0)\define1$, $\Delta_{-1} \define \emptyset$ and $\rc(\Delta_{-1})\define0$. 
\begin{lemma}
  \label{lem:subadditive}
  Let $k, \ell \in \Z_{\ge 0}$ with $k>0$ or $\ell>0$.  Then, for non-empty finite sets $X \subset \Z^k$ and $Y \subset \Z^\ell$ with $\dim(X)=k$ and $\dim(Y)=\ell$, one has 
  $\rc(X \ast Y ) \leq \rc(X) + \rc(Y)$.
\end{lemma}
\begin{proof}
  We may assume w.l.o.g.\ that both~$X$ and~$Y$ admit a relaxation, i.e.,
  $\rc(X)$ and~$\rc(Y)$ are finite, as otherwise the statement holds trivially.
  We first consider the case~$k> 0$, $\ell>0$. 
  Without loss of generality we may assume that both $X$ and $Y$ contain the null vector. 
  Let~$Ax \leq b$ and~$Cy \leq d$ be inequality systems defining
  relaxations~$P$ and~$Q$ of~$X$ and~$Y$, respectively, having $\rc(X)$ and $\rc(Y)$ inequalities, respectively. 
  Since both~$X$ and~$Y$ contain the origin, $b$ and~$d$
  are non-negative vectors.
  In particular, there needs to be at least one strictly positive entry
  in~$b$ and~$d$, as otherwise, $P$ or $Q$ would be  a full-dimensional cone and would contain infinitely many integer points. 

 To verify the inequality, it suffices to  show that the polyhedron
  \[
    R = \{ (x,y,z) \in \R^k \times \R^\ell \times \R :
    Ax \leq b\cdot (1 - z),\; Cy \leq d\cdot z\},
  \]
  defined by $\rc(X) + \rc(Y)$ inequalities, 
  is a relaxation of~$X \ast Y$.

  Let~$(x,y,z) \in R \cap \Z^{k+\ell+1}$.
  If~$z = 0$, then~$y$ is contained in the recession cone~$\rec(Q)$ of~$Q$.
  Since~$Q$ is a relaxation of~$Y$, we conclude~$y = \mathbf{0}$
  as otherwise there are infinitely many integer points in~$Q$.
  Moreover, setting~$z=0$ gives the original system~$Ax \leq b$, i.e.,
  $x \in X$ follows.
  Analogously, if~$z=1$, then~$x \in \rec(P)$ and~$y \in Y$.
  Thus, $(x,y,z) \in X \ast Y$ follows in both cases.

  If~$z \leq -1$, then $\{y\in\R^\ell : Cy \leq d\cdot z\}$ is a subset of $\rec(Q)$, since
  decreasing the right-hand side makes the system more restrictive as~$d$
  is non-negative.
  Since there exists at least one entry of~$d$ that is positive, $Cy \leq
  d\cdot z$ excludes the null vector if~$z < 0$.
  Hence, $\{y\in\R^\ell : Cy \leq d\cdot z\}$ does not contain any integer point.
  Since we can argue analogously for~$z \geq 2$ and the system $Ax \leq b\cdot(1-z)$,
  we conclude~$R \cap \Z^{k+\ell+1} \subseteq X \ast Y$. The reverse inclusion $X \ast Y \subseteq R \cap \Z^{k+\ell+1}$ is straightforward. 
  
  When $k$ or $\ell$ is zero, the argument is similar. Without loss of generality, let $\ell=0$. Then $Y = \Delta_0$ and  $X \ast Y = (X \times \{0\}) \cup \{e_{k+1}\}$. Again, we can assume that $\mathbf{0} \in X$ and fix a system $Ax \le b$ of $\rc(X)$ inequalities that defines a relaxation of $X$. One can see that the system $Ax \le b \cdot (1-z)$, $z \ge 0$ of $\rc(X)+1 = \rc(X)+\rc(Y)$ inequalities  defines a relaxation of  $X \ast Y$ following the proof strategy we used in the case $k>0,\ell>0$. 
\end{proof}
Using Lemma~\ref{lem:subadditive}, Corollary~\ref{cor:dimd} follows easily.
\begin{proof}[Proof of Corollary~\ref{cor:dimd}]
	  Note that~$\rc(\Delta_{k+\ell+1}) = \rc(\Delta_k \ast \Delta_\ell)$
	because~$\Delta_{k+\ell+1}$ and~$\Delta_k \ast \Delta_\ell$ are
	unimodularly equivalent. 

	Let $m$ and $r$ be the quotient and the remainder of the division of $d+1$ by $6$. By the above observation, $\Delta_d$ is equivalent to the free join of $m$ copies of $\Delta_5$   and one copy of $\Delta_{r-1}$ for $r>0$. Thus, using Lemma~\ref{lem:subadditive} and Theorem \ref{thm:dim5}, we obtain $\rc(\Delta_d) \le m \rc(\Delta_5) + \rc(\Delta_{r-1}) \le 5 m + r$.    
\end{proof}

\begin{remark}
	Combining Proposition~\ref{prop:mon:rc:delta} and the inequality $\rc(\Delta_{d+1}) \le \rc(\Delta_d \ast \Delta_0) \le \rc(\Delta_d) + 1$, which follows from Lemma~\ref{lem:subadditive}, we conclude that, for any given $d \in \Z_{>0}$, the value  $\rc(\Delta_{d+1})$ either stays equal to $\rc(\Delta_d)$ or grows with respect to $\rc(\Delta_d)$ by exactly one unit. Thus, for the determination of the exact behavior of $\rc(\Delta_d)$ in $d$, the frequency of the jumps by one unit would have to be estimated from below and above. 
\end{remark}

As a consequence of our proof approach for Theorem~\ref{thm:dim5}, we also obtain the following. 

\begin{cor}
	For every dimension $d \in \Z_{\ge 5}$, up to identification with respect to unimodular transformations, there exist infinitely many sets $X \subset \Z^d$ with $\dim(X)=d$ and~$\rc(X) < d+1$. 
\end{cor} 
\begin{proof}
	We first settle the case $d=5$. For each $a \in \Z_{>0},$ we consider the one-parametric generalization  	\[
	X_a \define \{\mathbf{0},e_1,e_2,e_3, (1,0,1,1,0), (0,1,1,0,a) \}
	\]
	of the set $\Delta$, with $\Delta=X_a$  for $a=1$. The projection of $X_a$ along the irrational line $\ell_a$ spanned by  $(0,0,0, 1, a \sqrt{2})$ under the map $\pi_a\colon \R^5 \to \R^4$,
 	\[
		\pi_a(x_1,x_2,x_3,x_4,x_5) = (x_1,x_2,x_3,x_4 - \frac{1}{a \sqrt{2}} x_5). 
	\]	
	is the same set $\Delta'$ as the projection of $\Delta$ under the projection map $\pi$ considered in the analysis of $\Delta$. It is clear that $\pi_a$ maps $\Z^5$ into $\Z^3 \times \R$. 
	Consequently, in the same way as we showed that $\pi^{-1}(P_{\nicefrac{1}{8}})$ is a five-facet relaxation of $\Delta$ within $\Z^5$ relying on Observation~\ref{obs:proj} and Proposition~\ref{prop:P}, 
		we can also show that $\pi_a^{-1}(P_{\nicefrac{1}{8}})$  is a five-facet relaxation of~$X_a$ within~$\Z^5$ relying on a straightforward generalization of Observation~\ref{obs:proj} (with $\Delta$ and~$\pi$ replaced by~$X_a$ and~$\pi_a$, respectively) and Proposition~\ref{prop:P}. 
		
	Distinct sets $X_{a'}$ and $X_{a''}$ with $a',a'' \in \Z_{>0}$ and $a' \ne a''$ do not coincide up to unimodular transformations when $a' \ne a''$, since $X_a$ is the vertex set of a simplex of volume $\frac{1}{5!} a$ and unimodular transformations preserve volume. This completes the proof of the assertion for $d=5$. 
	
 	For higher dimensions $d \ge 6$, we use 
 	\(
 			Y_a \define X_a \ast \Delta_{d-6}.
 	\)
 	The set $Y_a$ is the vertex set of a simplex of volume $\frac{1}{d!} a$. Thus, all of the sets $Y_a$ with $a \in \Z_{>0}$ are pairwise distinct up to unimodular transformations. Lemma~\ref{lem:subadditive} yields $\rc(Y_a) \le \rc(X_a) + \rc(\Delta_{d-6}) \le 5 + d-5 = d < d+1$. 
\end{proof}

\subsection{The Core Steps to Find the Counterexample}

The core steps to find our counterexample consist of a projection of
(a unimodular transformation of) $\Delta_5$ to a lower
dimensional space and  a determination of  a $(3,1)$-mixed-integer relaxation of the
projection.
As we will see below, the continuous coordinate of such a relaxation can be
interpreted as a lifting from a 3-dimensional space.
For this reason, to be able to generalize our proof strategy to
prove Theorem~\ref{thm:main}, we provide two lemmas that explain how the
relaxation complexity behaves under these two operations: projection and
lifting.
Afterwards, we will review our approach to prove Theorem~\ref{thm:dim5} in
the light of these two lemmas, which will motivate our strategy to prove
Theorem~\ref{thm:main}.

\begin{lemma}[On injective projections]\label{lem:inj:proj}
  Let $X \subseteq Z \subseteq \R^d$ and let  $\pi\colon \R^d \to \R^n$ be
  an affine map such that $\pi|_Z$ is injective.
  Then,
  \[
    \rc(X,Z) \le \rc(\pi(X),\pi(Z)).
  \] 
\end{lemma}

\begin{proof}
  Let $k = \rc(\pi(X),\pi(Z))$.
  We may assume that $k$ is finite, as
  otherwise the assertion is trivial.
  Consider a polyhedron $Q$ with $k$ facets, which is a relaxation of
  $\pi(X)$ within~$\pi(Z)$, which means that $Q \cap \pi(Z) = \pi(X)$
  holds.
  It is clear that the polyhedron $P = \pi^{-1}(Q)$ has at most $k$
  facets.	
  For verifying the asserted inequality, it suffices to show that $P$ is a
  relaxation of $X$ within $Z$, which means that $P \cap Z = X$.
  
  To check the inclusion $P \cap Z \subseteq X$, we consider an arbitrary
  $y \in P \cap Z$.
  From the definition of $P$, we see that $\pi(y) \in Q \cap \pi(Z)$.
  In view of $Q \cap \pi(Z) = \pi(X)$, we obtain~$\pi(y) \in\pi(X)$.
  As $\pi|_Z$ is injective and $X \subseteq Z$, we conclude that $y \in X$.
	
  For showing the converse inclusion $X \subseteq P \cap Z$, we consider an
  arbitrary  $y \in X$.
  Since~$X \subseteq Z$, we have $y \in Z$.
  The condition $y \in P$, which remains to be verified, is equivalent to
  $\pi(y) \in Q$.
  Since $Q \cap \pi(Z) = \pi(X)$ and $\pi(y) \in \pi(X)$, we see that
  $\pi(y)$ is indeed in~$Q$.
\end{proof} 

Next, we show how we can derive upper bounds on the~$(k,1)$-mixed-integer
relaxation complexity of finite sets by means of lifting.
To formalize this, we introduce, for $h\colon X \to \R$,
\begin{align*}
	\lift_h(X) & \define \setcond{ (x,h(x)) }{ x \in X },
	\\ \clift_h(X) & \define \conv(\lift_h(X)),
\end{align*}
and call these sets the \emph{lift} (resp.\ \emph{convex lift}) of $X$ with
respect to the \emph{height function} $h$.

Consider a $(k+1)$-dimensional polytope $P$ in $\R^{k+1}$.
The facet-defining inequalities for this polytope are inequalities in
$(x,y) \in \R^k \times \R$, which can be brought into one of the three
forms $y \le u^\top x + v $, $y \ge u^\top x +v $ or $0 \le u^\top x + v$, where $u \in \R^k$ and $v \in \R$. 
We call the respective facets of $Q$ the \emph{upper facets}, the \emph{lower
  facets}, and the \emph{lateral facets}.

\begin{prop}\label{prop:relaxation}
    Let $P\subset \R^{k+1}$ be a full-dimensional polytope, $h\colon T\to
    \R$  be a function from a finite non-empty set $T\subset \Z^k$.
    Then $P$ is a relaxation of $\lift_h(T)$ within $\Z^k\times \R$ if and
    only if
    \begin{enumerate}
    \item\label{it:prop1} every $p\in \lift_h(T)$ is contained in an upper
      and lower facet of $P$, and
    \item\label{it:prop2} the projection of $P$ onto the first $k$
      components is a relaxation of $T$ within $\Z^k$.
    \end{enumerate}
\end{prop}

\begin{proof}
First, suppose that $P$ is a relaxation of $\lift_h(T)$ within $\Z^k\times \R$. If, by contradiction, there is $(x,h(x))\in \lift_h(T)$ that is not contained in any upper facet of $P$, then there is some~$y>h(x)$ such that $(x,y)\in P$, contradicting the fact that $P\cap \Z^k\times \R = \lift_h(T)$. Arguing similarly for lower facets, one checks Condition \ref{it:prop1}. Condition \ref{it:prop2} follows immediately from the definition of relaxation.

For the opposite direction, let $P$ satisfy Conditions~\ref{it:prop1} and~\ref{it:prop2}. Condition~\ref{it:prop1} easily implies that $P\cap (\Z^k\times \R) \supseteq \lift_h(T)$. For the reverse inclusion, let $(x,y)\in P\cap (\Z^k\times \R)$. We need to show that $(x,y) \in \lift_h(T)$. By Condition~\ref{it:prop2}, one has $x \in T$. If $(x,y)$ were not in $\lift_h(T)$, one would have $y\neq h(x)$.
Then, if $y > h(x)$ (resp. $y<h(x)$), we conclude that $(x,y)$ violates every upper facet (resp.\ lower facet) inequality of $P$, in particular those that are active on $(x,h(x))$, a contradiction. 
\end{proof} 

Motivated by a free-sum operation for polytopes \cite[15.1.3]{henk200416}, we
introduce the \emph{free sum} of~$X \subseteq \Z^m$ and $Y \subseteq \Z^n$
as
\[
  X \oplus Y \define \setcond{ (x,\mathbf{0}) }{x \in X} \cup \setcond{(\mathbf{0},y)}{y \in Y}.   
\]
This operation will be useful, because $\Delta_k \oplus \Delta_m = \Delta_{k+m}$.

\begin{lemma} \label{lem:X:oplus:Delta} 
  Let $X, Y$ be two disjoint finite subsets of $\Z^k$ with
  $\mathbf{0} \in X$ and $|Y|=m$.
  Let~$h\colon X \cup Y \to \R$ be such that $h(p) =0$ for every $p \in
  X$ and the values $h(p)$ with $p \in Y$ are linearly independent
  over $\Q$.
  Then 
  \(
  \rc(X \oplus \Delta_m) \le \rc(\lift_h(X \cup Y), \Z^k \times \R). 
  \)
\end{lemma}
      
\begin{proof}
  Let $Y = \{p_1,\ldots,p_m\}$. 
  The map $\phi\colon \R^k\times \R^m \rightarrow \R^k\times \R^m$ which is
  defined as $\phi(a,b) \define \left(a + \sum_{i=1}^m b_i p_i, b\right)$
  is unimodular and it satisfies 
  \[
    \phi(X \oplus \Delta_m)
    =
    \setcond{ (x,\mathbf{0}) }{x \in X} \cup \setcond{ (p_i , e_i )  }{i
      \in \{1,\ldots,m\}}.
  \]
  We combine this map with the map $\psi(a,b) \define \left( a,
    \sum_{i=1}^m h(p_i) b_i \right)$.
  Since $h(p_1),\ldots,h(p_m)$ are linearly independent over
  $\Q$, the map $\psi|_{\Z^{k+m}}$ is injective.
  Consequently $(\psi \circ \phi)|_{\Z^{k+m}}$ is injective as well, so we
  obtain
  \begin{align*}
    \rc(X \oplus \Delta_m, \Z^{k+m}) & \le \rc( \psi(\phi( X \oplus \Delta_m)), \psi(\phi(\Z^{k+m})) & & \text{(by Lemma~\ref{lem:inj:proj})}. 
    \\ & = \rc(\lift_h(X \cup Y), \psi( \Z^{k+m} ))
    \\ & \le \rc(\lift_h(X \cup Y), \Z^k \times \R). & & \text{(since $\psi( \Z^{k+m} ) \subseteq \Z^k \times \R$)}
  \end{align*} 
\end{proof} 

Having Lemmas~\ref{lem:inj:proj} and \ref{lem:X:oplus:Delta} in mind, the summary of the proof of the upper bound $\rc(\Delta_5) \le 5$ is as follows.
First, we find it convenient to go from $\Delta_5$ to
$\Delta$.
One can represent these two sets as columns of two
matrices (see Figures~\ref{fig:mat1} and~\ref{fig:mat2}), which easily allows
to check that they are unimodularly equivalent.
Second, we apply the projection $\pi$ on $\R^4$ along line $\ell$, as defined in
Observation \ref{obs:proj}.
This corresponds to applying Lemma~\ref{lem:inj:proj}, where~$X=\Delta$,
$Z=\Z^5$.
The points in~$\pi(\Delta)$ are listed as columns of the matrix in
Figure~\ref{fig:mat3}.
Now, in light of Lemma~\ref{lem:X:oplus:Delta}, we see the points
in~$\pi(\Delta)$ as coming from a lift of points in $\Z^3$.
Let $X=\Delta_3$ and $Y=\{(1,0,1), (0,1,1)\}$, and $h$ such that $h(p)=0$
for $p\in X$, $h(1,0,1)= 1$, and $h(0,1,1)= -1/\sqrt{2}$.
Then, by Lemma~\ref{lem:X:oplus:Delta}, we have that $\rc(\Delta_5)=\rc(X
\oplus \Delta_2) \le \rc(\lift_h(X \cup Y), \Z^3 \times \R)$.

\begin{figure}[tb]
    \centering
    \subfloat[]{$\left(
\begin{array}{cccc|cc}
   0 & 1 & 0 & 0 & 0 & 0\\
   0 & 0 & 1 & 0 & 0 & 0 \\
   0 & 0 & 0 & 1 & 0 & 0 \\
   \hline
   0 & 0 & 0 & 0 & 1 & 0 \\
   0 & 0 & 0 & 0 & 0 & 1 \\
\end{array}\right)$ \label{fig:mat1}}
    \subfloat[]{$\left(
\begin{array}{cccc|cc}
   0 & 1 & 0 & 0 & 1 & 0\\
   0 & 0 & 1 & 0 & 0 & 1 \\
   0 & 0 & 0 & 1 & 1 & 1 \\
   \hline
   0 & 0 & 0 & 0 & 1 & 0 \\
   0 & 0 & 0 & 0 & 0 & 1 \\
\end{array}\right)$\label{fig:mat2}}
\subfloat[]{$\left(\begin{array}{cccc|cc}
   0 & 1 & 0 & 0 & 1 & 0\\
   0 & 0 & 1 & 0 & 0 & 1 \\
   0 & 0 & 0 & 1 & 1 & 1 \\
   \hline
   0 & 0 & 0 & 0 & 1 & -\frac{1}{\sqrt{2}} \\
\end{array}\right)$\label{fig:mat3}}
\caption{Visualization of the reductions in the proof of Theorem~\ref{thm:dim5}.}
\label{fig:mat}
\end{figure}

In order to obtain Theorem \ref{thm:dim5}, it now remains to show that $P$,
defined as in Section~\ref{sec:counterex}, is a relaxation of $\lift_h(X
\cup Y)$ within $\Z^3 \times \R$.
One can do that by setting $T=X\cup Y$ and verifying
Conditions~\ref{it:prop1} and~\ref{it:prop2} in Proposition~\ref{prop:relaxation}.
Condition~\ref{it:prop1} can be easily verified by the reader (notice that
inequalities \eqref{constr:1},\eqref{constr:2},\eqref{constr:3} define
upper facets of $P$, and inequalities \eqref{constr:4},\eqref{constr:5}
define lower facets).
Verifying Condition \ref{it:prop2} can be done directly via Fourier-Motzkin elimination, as mentioned in the proof of  Theorem~\ref{thm:dim5}.

In summary, our proof of Theorem~\ref{thm:dim5} consists of three steps:
\begin{enumerate*}
\item[\textbf{project}] (a unimodular transformation of) the simplex to a
  lower dimensional space;
\item[\textbf{lift}] the projection into a space with one continuous coordinate;
\item[\textbf{relax}] the mixed-integer point set by finding a small relaxation.
\end{enumerate*}
To prove Theorem~\ref{thm:main}, we once again utilize this project-lift-relax
strategy.
That is, the steps of our proof are:
\begin{enumerate}[label=(S\arabic*)]
\item\label{S1} Find a suitable unimodular
 copy  of $\Delta_d$  and
  project this copy onto~$\Z^k$ via a map~$\pi$.
\item\label{S2} Construct a suitable height function~$h\colon \pi(\tilde{\Delta}_d) \to \R$ 
\item\label{S3} Derive a relaxation~$P$
  of~$\lift_h(\pi(\tilde{\Delta}_d))$ within~$\Z^k \times \R$.
\end{enumerate}
By a variation of Lemma~\ref{lem:X:oplus:Delta}, which we formulate below as Lemma~\ref{lem:X:oplus:Delta:special}, 
the number of facets of~$P$ will
provide an upper bound on~$\rc(\tilde{\Delta}_d)$ and thus
on~$\rc(\Delta_d)$ since~$\Delta_d$ and~$\tilde{\Delta}_d$ are unimodularly
equivalent, provided the height function~$h$ satisfies all assumptions.
In the following subsections, we discuss each of these three steps in
detail.

\section{The Proof of Theorem~\ref{thm:main}}
\label{sec:thm3}

This section is devoted to proving Theorem~\ref{thm:main}.

\subsection{Project: Selecting the Right Simplex and Its Projection}
\label{sec:project}

Step~\ref{S1} requires us to select a set~$\tilde{\Delta}_d$ that is unimodularly equivalent to~$\Delta_d$ as well as a
projection~$\pi\colon \R^d \to \R^k$.
To this end, consider the following~$d\times (d+1)$ matrix
\[
M = \begin{pmatrix}
 \mathbf{0}_k & I_k & B \\
  \mathbf{0}_{d-k} & \mathbf{0}_{d-k, k} & I_{d-k}
\end{pmatrix},
\]
where $I_k$ denotes
the identity matrix of order $k$, and $B$ is the $k \times (d-k)$ matrix
whose columns are all the vectors in $\{0,1\}^k$ with at least two
1-entries.
Note that this definition requires~$d = 2^k - 1$.
Although this assumption might look restrictive, we can reduce the proof of
Theorem~\ref{thm:main} to this special case as we will see later.

Let $\tilde\Delta_d\subset \R^d$ be the set of columns of $M$.
It is easy to see that~$\tilde{\Delta}_d$ and~$\Delta_d$ are unimodularly
equivalent.
We select~$\pi$ to be the orthogonal projection of~$\R^d$ onto the
first~$k$ coordinates.
In the spirit of Lemma~\ref{lem:X:oplus:Delta}, we then let $X\subset
\Z^k$ be the set containing the first $k+1$ columns of $M$ (truncated to
the first $k$ coordinates), i.e., $X =\Delta_k$, and $Y\subset \Z^k$ be the set containing the
truncation of the other columns of $M$, so that $T \define X\cup Y=
\{0,1\}^k$.

\subsection{Lift: Finding the Right Height Function}
\label{sec:lift}

Our next goal is to find a suitable height function~$h$ for Step~\ref{S2}.
In view of Step~\ref{S3}, we need to select~$h$ in such a way that the we
can easily derive a mixed-integer relaxation of~$\lift_h(T)$.
By Proposition~\ref{prop:relaxation}, such a relaxation is
given by taking the~$2k$ box constraints $z \in [0,1]^k \times \R$ and
selecting upper and lower facets of~$\clift_h(Y)$ that
cover all points~$(z,h(z))$ for~$z \in T$.
Our strategy for finding a small relaxation is thus
\begin{enumerate}[label=(S\arabic*a)]
  \setcounter{enumi}{2}
\item\label{S3a}
  to find a small covering of the points~$(z, h(z))$ for~$z \in T$
  consisting of upper and lower facets, respectively.
\end{enumerate}

To be able to quantify the minimum size of a covering by upper and lower
facets, respectively, we need to derive combinatorial properties of
the facet structure of~$\clift_h(T)$.
In our analysis in the next subsection, it will thus be convenient to have
a function~$h$ that satisfies~$h(z) \in \Q$ for all~$z \in T$.
In view of Lemma~\ref{lem:X:oplus:Delta}, however, we require
that the values~$h(y)$ for all~$y \in Y$ are linearly independent
over~$\Q$.
To achieve this, we will slightly perturb~$h$, but note that this may change the
combinatorial structure of the facets of~$\clift_h(T)$.
As we show next, simplicial facets of~$\clift_h(T)$ remain simplicial after
small perturbations of~$h$.
Thus, we can still pursue our strategy if we restrict ourselves to
cover the points~$(z, h(z))$ in Step~\ref{S3a} by simplicial facets.

Therefore, we define the simplicial upper (resp.\ lower) covering numbers
of $\lift_h(T)$, denoted by $\sucn_h(T)$ (resp. $\slcn_h(T)$), as the
minimum number of simplicial upper (resp.\ lower) facets that cover all
vertices of the upper (resp.\ lower) part of  $\clift_h(T)$.
If no covering by simplicial facets exist, the respective number is set to
be $\infty$.

\begin{lemma}[On semi-continuity] \label{lem:semicont}
  Let $T$ be an $m$-dimensional subset of $\{0,1\}^m$.
  Then, for every $h \in \R^T$ which is not a restriction of an affine
  function, there exists an open neighborhood $U \subseteq \R^T$ of $h$
  with the property that for every $\tih \in  U$  one has
  \begin{align*}
    \sucn_{\tih}(T) & \le \sucn_h (T), 
    \\	\slcn_{\tih}(T) & \le \slcn_{h} (T).
  \end{align*}
\end{lemma} 

\begin{proof}
  We only prove the first inequality, since the proof of the second inequality is completely analogous. 
  Every simplicial upper facet of $Q\define \clift_h(T)$ is given by
  \[
    F_h
    \define
    \conv \left\{
      \begin{pmatrix}
        v_0 \\
        h(v_0)
      \end{pmatrix},
      \ldots,
      \begin{pmatrix}
        v_m \\
        h(v_m)
      \end{pmatrix}
    \right\},
  \]
  where $v_0,\ldots,v_m \in T$ are affinely independent.
  By standard linear algebra, with an appropriate ordering of
  $v_0,\ldots,v_m$, the respective facet-defining inequality for $Q$ is
  given by
  \begin{equation} \label{det:ineq}
    \det 
    \begin{pmatrix} 
      1 & \cdots & 1 & 1 \\
      v_0 & \cdots & v_m & x \\
      h(v_0) & \cdots & h(v_m) & y 
    \end{pmatrix}
    \le 0. 
  \end{equation}
  By expanding the determinant with respect to the last column, one can see
  that  the fact that $F_h$ is an upper facet can be expressed as
  \[
    \det
    \begin{pmatrix}
      1 & \cdots & 1 \\
      v_0 & \cdots & v_m
    \end{pmatrix}
    > 0\]
  and the  strict validity of the inequality on~$\vert(Q) \setminus \setcond{\binom{v_i}{h(v_i)} }{i\in\{0,\ldots,m\}}$:
  \begin{align} \label{det:strict:ineq}
    \underbrace{ 
    \det 
    \begin{pmatrix} 
      1 & \cdots & 1 & 1 
      \\ v_0 & \cdots & v_m & x
      \\ h(v_0) & \cdots & h(v_m) & h(x)
    \end{pmatrix} }_{\eqqcolon\phi_{h}(x)} & < 0  & & \text{for all} \ x \in T \setminus \{v_0,\ldots,v_m\}. 
  \end{align} 
  Since $\phi_h(x)$ is linear in $h$, the set 
  \[
    H \define \setcond{\tih  \in \R^T}{\phi_{\tih}(x) < 0 \text{ for all
      } x \in T \setminus \{v_0,\ldots,v_m\}}
  \]
  is open (it is the interior of a polyhedron) and it satisfies $h \in
  H$.
  By construction, for every $\tih \in H$, the set $F_\tih$ is an upper facet of
  $\clift_{\tih}(T)$.

  Thus, whenever $t$ simplicial upper facets $F_h^1,\ldots,F_h^t$ 
  cover the vertex set of $\clift_h(T)$ we see there exist $t$
  respective open sets $H^1,\ldots,H^t$ with $h \in H^i$ for every
  $i\in\{1,\ldots,t\}$.
  It follows that, for every $\tih$ in the open set $H^1\cap \cdots \cap H^t$,
  the sets $F_\tih^1,\ldots,F_\tih^t$ are simplicial upper facets of $\clift_\tih(T)$
  that cover the vertex set of $\clift_\tih(T)$.
  In particular, the latter holds for every $\tih$ in a sufficiently small
  neighborhood of $h$.
\end{proof} 

Another difficulty arises in the tools that we have available for
Step~\ref{S2}, because Lemma~\ref{lem:X:oplus:Delta} requires a height
function~$h$ with~$h(x) = 0$ for all~$x \in X$.
We were not able to find such a height function that allows us to find a
small relaxation in Step~\ref{S3a}.
For this reason, we need a strengthening of Lemma~\ref{lem:X:oplus:Delta},
which we give in Lemma~\ref{lem:X:oplus:Delta:special} below.
In preparation for this, we provide the following proposition, which is
actually a classical observation used in the theory of regular polyhedral
subdivisions.
For the sake of completeness, we give a short proof.

\begin{prop} \label{prop:affine:diff} 
  Let $X \subseteq \Z^m$ and $h, \tih\colon X \to \R$ be such that $h(x) -
  \tih(x)=f(x)$ for every $x\in X$, where  $f\colon\R^m\to \R$ is an
  affine function.
  Then 
  \[
    \rc( \lift_h(X), \Z^m \times \R) = \rc(\lift_\tih(X), \Z^m \times \R). 
  \]
\end{prop} 
\begin{proof}
  The affine map $\phi(x,y) = (x, y + f(x))$ maps both $\R^m\times\R$ and
  $\Z^m \times \R$ bijectively onto themselves and it satisfies
  $\phi(x,\tih(x)) = (x,h(x))$, which means that $\phi$ maps
  $\lift_{\tih}(X)$ bijectively onto $\lift_h(X)$.
  Thus every relaxation of $\lift_{\tih}(X)$ within $\Z^m \times \R$ is
  bijectively  mapped to a relaxation of $\lift_h(X)$ within $\Z^m \times
  \R$ and vice versa.
  This yields the desired assertion. 
\end{proof} 

\begin{lemma} \label{lem:X:oplus:Delta:special}
  Let $X$ be an affinely independent subset of $\Z^k$ with $\mathbf{0} \in
  X$, let $Y$ be a finite $m$-element subset of $\Z^k$ disjoint with $X$,
  and let $h\colon X \cup Y \to \R$ be such that \
  \begin{enumerate}[label=(\alph*)]
  \item  $h(x) \in \Q$ for every $x \in X$ and
  \item  the values $h(y)$ with $y \in Y$ together with $1$ are linearly
    independent over $\Q$.
  \end{enumerate} 
  Then $\rc(X \oplus \Delta_m) \le \rc(\lift_h(X \cup Y), \Z^k \times \R)$. 
\end{lemma} 
\begin{proof} 
  Since $X$ is affinely independent, there exists an affine function $f$
  that coincides with~$h$ on $X$.
  Since $h(x) \in \Q$ for every $x \in X$, we can choose $f$ with
  rational coefficients.
  Then the function $h - f$ is equal to $0$ on $X$. 

  We now verify that the assumptions of Lemma~\ref{lem:X:oplus:Delta} are
  fulfilled by checking that the values $h(y) - f(y)$  with $y \in Y$ are
  linearly independent over $\Q$.
  By the choice of $f$, we have $f(y) \in \Q$ for every $y \in Y$.
  Hence, since the values $h(y)$ with $y \in Y$ together with $1$ are
  linearly independent over $\Q$, it follows the values $h(y)-f(y)$ with $y
  \in Y$ are linearly independent over $\Q$.
  Thus, using Lemma~\ref{lem:X:oplus:Delta} and then
  Proposition~\ref{prop:affine:diff}, we obtain
  \[
    \rc(X \oplus \Delta_m)
    \le
    \rc(\lift_{h-f}(X \cup Y), \Z^k \times \R)
    =
    \rc(\lift_h(X \cup Y, \Z^k \times \R). 
  \]
\end{proof} 

Equipped with these tools, we are able to provide a statement that allows
us to bound the relaxation complexity in terms of simplicial upper and
lower covering numbers.

\begin{lemma}\label{lem:X:oplus:Delta:simp}
  Let $X\subseteq \{0,1\}^k$,
  let $Y$ be an $m$-element subset of $\{0,1\}^k$
  disjoint with~$X$, and let $h\colon X \cup Y \to \R$ be any
  function which is not a restriction of an affine function and satisfies
  $h(x) \in \Q$ for all $x \in X$.
  Then 
  \[
    \rc(X \oplus \Delta_m) \le \rc(X \cup Y)  + \sucn_h(X \cup Y ) + \slcn_h(X \cup Y). 
  \]
\end{lemma} 

\begin{proof}
  One can approximate $h$ arbitrarily well by $\tih\colon X \cup Y \to \R$
  for which $\tih(x) = h(x)\in \Q$ for every $x \in X$ and the values
  $\tih(y)$ with $y \in Y$, together with 1, are linearly independent over
  $\Q$.
  Thus, the assertion follows by combining Proposition
  \ref{prop:relaxation} with $T=X\cup Y$ and
  Lemmas~\ref{lem:X:oplus:Delta:special} and~\ref{lem:semicont}.
\end{proof}

Recall from the previous subsection that we want to apply a lifting to the
set~$T = X \cup Y$ with~$X = \Delta_k$ and~$Y = \{0,1\}^k \setminus X$.
Due to Lemma~\ref{lem:X:oplus:Delta:simp}, we can use in Step~\ref{S2} of our
strategy any height function~$h$ such that~$h|_T$ is not the restriction of an affine
function and such that~$h(x) \in \Q$ for all~$x \in X$.
Towards proving Theorem~\ref{thm:main}, we select
\[
  h\colon \{0,1\}^k \to \R,\qquad
  h(x_1,\dots,x_k) = (2x_k - 1) \cdot r(x_1,\dots,x_{k-1}),
\]
where~$r\colon \{0,1\}^{k-1} \to \R$ is defined by $r(x_1,\dots,
x_{k-1})=(\sum_{i=1}^{k-1} x_i)^2$.
\begin{remark}
  Function~$r$ is the height function that defines the so-called staircase
  triangulation of~$[0,1]^{k-1}$, see, e.g., \cite[Sec.~16.7.2]{LeeSantos}.
  That is, the statements about simplicial facets of~$\lift_h(X \cup Y, \Z^k
  \times \R)$, which we will make in the following section, can be related
  to a triangulation of~$[0,1]^{k-1}$ if we fix~$x_k$ to~0 or~1.
\end{remark}

\subsection{Relax: Constructing a Small Mixed-Integer Relaxation}
\label{sec:relax}

To complete Step~\ref{S3a}, we use the height function~$h$ from the
preceding subsection and bound $\sucn_h(\{0,1\}^k)$ and
$\slcn_h(\{0,1\}^k)$ from above.
First, we show that the two values are equal, i.e., it is sufficient to
investigate the upper or lower facets.

\begin{prop}\label{prop:up=low}
  One has
  \[
    \slcn_h(\{0,1\}^k) = \sucn_h(\{0,1\}^k).
  \]
\end{prop}
\begin{proof}
    	The affine map $\phi(x_1,\ldots,x_k,y) \define (x_1,\ldots,x_{k-1},1-x_k, -y)$ is an affine involution on $\R^{k+1}$. Furthermore, $\phi$ can be restricted to an involution on $\lift_h(\{0,1\}^k)$ since 
	\begin{align*} 
		& \phi \bigl ( x_1,\ldots,x_{k-1},x_k, (2 x_k - 1) \cdot r(x_1,\ldots,x_{k-1}) \bigr) 
		\\ = & \bigl ( x_1,\ldots, x_{k-1},1-x_k, - (2 x_k - 1) \cdot r(x_1,\ldots,x_{k-1}) \bigr),
	\\ =  & \bigl(x_1,\ldots,x_{k-1},  1- x_k , (2 (1-x_k) - 1) \cdot (r(x_1,\ldots,x_{k-1}) \bigr) 
	\\ \in & \lift_h(\{0,1\}^k)
	\end{align*} 
for each $(x_1,\ldots,x_k) \in \{0,1\}^k$. This implies that a subset $F\subseteq \clift_h(\{0,1\}^k)$ defines a facet of $\clift_h(\{0,1\}^k)$ if and only if the set $\phi(F)$ does. Now, in order to obtain the identity in the thesis it thus suffices to show that $F$ is an upper facet if and only if~$\phi(F)$ is a lower facet. Indeed, let $f\colon \R^{k+1} \to \R$ be the affine function defining an upper facet~$F$ of $\clift_h(\{0,1\}^k)$, so that $f(x)\leq 0$ is valid for each $x\in \clift_h(\{0,1\}^k)$, equality is attained for $x\in F$, and the coefficient of $x_{k+1}$ is positive. Then $f(\phi(x))\leq 0$ is valid for each $x\in \clift_h(\{0,1\}^k)$, with equality attained for $x\in \phi(F)$, and the coefficient of $x_{k+1}$ is negative, hence the facet~$\phi(F)$ given by the affine function $f(\phi(x))$ is a lower facet of~$\clift_h(\{0,1\}^k)$. The reverse implication holds by an analogous argument.
\end{proof}

Proposition \ref{prop:up=low} allows us to focus on the upper facets of $\clift_h(\{0,1\}^k)$. We will consider two families of simplicial (upper) facets: one will cover vertices with $x_k=0$ and the other will cover vertices with $x_k=1$.
For simplicity we denote a vertex of $\clift_h(\{0,1\}^k)$ by $(x,x_{k+1})$, where $x\in \{0,1\}^k$ and $x_{k+1}=h(x)$. We now define the facets that we will use in our covering.
Let $\pi$ be a permutation of $[k-1]$, $B\subseteq [k-1]$, $b=|B|$, and let $\pi'\colon[k-1-b] \rightarrow [k-1]\setminus B$ be a bijection. Let
  \[F_\pi=\conv\left(\{(\mathbf{0},0), (e_k,0)\}\cup\Big\{ (e_k + \sum_{i = 1}^t e_{\pi(i)}, t^2) : t\in\{1,\dots, k-1\}\Big\}\right), \]
  and
  \begin{align*}
    F_{B,\pi'}= \conv\Big(&\{(e(B),-b^2), (e(B) + e_k, b^2)\} \cup \{(e(B\setminus j),-(b-1)^2), j\in B\}\\
    &\cup\Big\{ (e(B)+e_k + \sum_{i = 1}^{t} e_{\pi'(i)}, (b+t)^2) : t\in\{1,\dots, k-1-b\}\Big\}\Big),
  \end{align*}
  where $e(B)=\sum_{i\in B} e_i$.

\begin{prop}\label{prop:facets}
  $F_\pi$ and $F_{B,\pi'}$ as defined above are simplicial upper facets of $\clift_h(\{0,1\}^k)$.
\end{prop}
\begin{proof}
    $F_\pi$ can be easily seen to contain exactly $k+1$ affinely independent vectors, hence in order to show that it is a facet, it suffices to exhibit a valid inequality for $\clift_h(\{0,1\}^k)$ that is satisfied at equality precisely by the  points of $F_\pi$. One can argue in the same way for $F_{B,\pi'}$.
    
    Consider the inequality
    \begin{align}
      x_{k+1}\leq \sum_{i=1}^{k-1} (2i-1)x_{\pi(i)}. \label{ineq:facet1}
    \end{align}
    Let $(x,x_{k+1})$ be a vertex of $\clift_h(\{0,1\}^k)$. If $x_k=0$, then \eqref{ineq:facet1} is trivially satisfied, as  the left-hand side is non-positive and the right-hand side would be non-negative, and it is satisfied with equality only if $x=\mathbf{0}$. Now, let $x_k=1$, hence $x_{k+1}=(\sum_{i=1}^{k-1} x_i)^2=\ell^2$ for some $\ell\in \Z_{\geq 0}$. Then the right-hand side of \eqref{ineq:facet1} is the sum of $\ell$ odd numbers, i.e., at least $\sum_{i=1}^\ell (2i-1)= \ell^2=x_{k+1}$, and equality holds exactly when $(x,x_{k+1})\in F_{\pi}$. 
    
    Now, consider the inequality
    	\begin{align}
    x_{k+1}\leq b^2-b - (2b-1)x(B)+\sum_{i=1}^{k-1-b} (2b+2i-1)x_{\pi'(i)} +2b^2x_k,  \label{ineq:facet0}
    \end{align}
    where $x(B)= \sum_{i\in B} x_i$ similarly as before.
    Let $h_1=x(B)$ and $h_2=x([k-1]\setminus B)$. Suppose~$x_k=1$. Then $x_{k+1}=(h_1+h_2)^2$ and the right-hand side of \eqref{ineq:facet0} becomes
    \[
   b^2-b - (2b-1)h_1+\sum_{i=1}^{k-1-b} (2b+2i-1)x_{\pi'(i)} +2b^2\geq b^2-b - (2b-1)h_1 + 2bh_2+ h_2^2+2b^2
    \]
    by lower bounding the summation $\sum_{i=1}^{k-1-b} (2i-1)x_{\pi'(i)}$ with the sum of the first $h_2$ odd numbers.
    The expression above is at least $(h_1+h_2)^2$ if and only if
    \[
    h_1^2 + 2h_1h_2 + (2b-1)h_1 - 2bh_2 \leq 3b^2-b.
    \]
    The latter is easily seen to hold, since $h_1\leq b$.
    
    The case $x_k=0$ follows from the previous case by noticing that the left-hand side of \eqref{ineq:facet0} decreases by $2(h_1+h_2)^2$ and the right-hand side decreases by $2b^2\leq 2(h_1+h_2)^2$.
    
    Finally, one can easily verify that the inequality is satisfied with equality exactly when~$(x,x_{k+1}) \in F_{B,\pi'}$, which concludes the proof.
\end{proof}

We now give our covering. The idea is to cover all vertices of $\clift_h(\{0,1\}^k)$ that have $x_k=1$ with facets of the form $F_\pi$, and the other vertices with facets of the form $F_{B,\pi'}$. 

\begin{prop}\label{prop:boundF1}
    There is a set $\Pi_k$ of $O(2^k/\sqrt{k})$ permutations on $[k-1]$ such that the facets in $\cF_1\define \{F_\pi : \pi \in \Pi_k\}$ cover all vertices of $\clift_h(\{0,1\}^k)$ with $x_k=1$.
\end{prop}

    To prove Proposition \ref{prop:boundF1}, we derive an upper bound on the number of simplices needed in a
triangulation of the hypercube~$[0,1]^k$.
Specifically, we consider the triangulation arising from permutations
contained in the symmetric group~$S_k$ of the set~$[k]$.
This triangulation is~$\cT_k = \{ \Delta^\pi : \pi \in S_k\}$
with~$\Delta^\pi = \{x \in \R^k : 1 \geq x_{\pi(1)} \geq \dots \geq
x_{\pi(k)} \geq 0\}$.

\begin{prop}\label{prop:covering_permutations}
  The minimum number of simplices from~$\cT_k$ needed to cover all vertices
  of~$[0,1]^k$ is~$\binom{k}{\lfloor \nicefrac{k}{2} \rfloor}$.
\end{prop}

\begin{proof}
  With each vertex~$x$ of~$[0,1]^k$, we associate the set~$A(x) = \{ i \in
  [k] : x_i = 1\}$.
  Moreover, let~$a(x) = \card{A(x)}$.
  Then, a simplex~$\Delta^\pi \in \cT_k$ contains~$x$ if and only
  if~$A(x) = \{\pi(1), \dots, \pi(a(x))\}$.
  The minimum number of simplices needed to cover all vertices of~$[0,1]^k$
  is thus equivalent to the minimum cardinality of a set~$\Pi \subseteq
  S_k$ of permutations such that, for each vertex~$x$ of~$[0,1]^k$, there
  is a permutation~$\pi \in \Pi$ with~$A(x) = \{\pi(1), \dots,
  \pi(a(x))\}$.

  To characterize such a minimum set~$\Pi$, consider the poset
  of all subsets of~$[k]$ ordered by inclusion.
  Note that there is a one-to-one correspondence between permutations~$\pi$ of~$[k]$
  and chains~$\emptyset \subseteq \{i_1\} \subseteq \{i_1,i_2\} \subseteq
  \dots \subseteq \{i_1, \dots, i_k\}$ in this poset via~$\pi(1) =
  i_1,\dots, \pi(k) = i_k$.
  Consequently, since subsets of~$[k]$ correspond to vertices of~$[0,1]^k$,
  the minimum number of simplices needed to cover all vertices of~$[0,1]^k$
  is the same as the minimum number of chains needed to cover all elements
  of the poset.
  By Dilworth's theorem, see, e.g., the book~\cite{VanLintWilson1993}, this number is the same as the length of a largest
  antichain in the poset.
  Since all subsets of~$[k]$ of  cardinality~$\lfloor \frac{k}{2} \rfloor$  are
  pairwise incomparable, the length of a largest antichain is at
  least~$\binom{k}{\lfloor \nicefrac{k}{2} \rfloor}$.
  Moreover, this bound is tight due to a result by Sperner~\cite{Sperner1928}.
\end{proof}

Proposition~\ref{prop:boundF1} thus follows, because~$\binom{k}{\lfloor
  \frac{k}{2} \rfloor} \in \Theta(\frac{2^k}{\sqrt{k}})$, see
\cite[Sect.~4]{odlyzko1995asymptotic}.

Now, to cover the vertices of $\clift_h(\{0,1\}^k)$ with $x_k=0$, we will select facets of the form $F_{B,\pi'}$. Notice that only the choice of the set $B$ determines which vertices with $x_k=0$ are in $F_{B,\pi'}$, hence the choice of the bijection $\pi'$ is irrelevant. We then fix one such bijection $\pi_B$ for each~$B$, and simply write $F_{B}\define F_{B,\pi_B}$.

\begin{prop}\label{prop:boundF0}
  There is a family $\mathcal{B}$ of $O(2^k\frac{\log(k)}{k})$ subsets of $[k-1]$ such that the facets in $\cF_0\define \{F_{B}: B \in \mathcal{B}\}$ cover all vertices of $\clift_h(\{0,1\}^k)$ with $x_k=0$.
\end{prop}
\begin{proof}
  Notice that a vertex of $\clift_h(\{0,1\}^k)$ with $x_k=0$ is in $F_B$ if and only if the support of $(x_1,\dots,x_{k-1})$ is either equal to $B$ or to a subset of $B$ of size $b-1$ (where again~$b=|B|$). Consider the Hasse diagram $D=(V,A)$ of the poset $2^{[k-1]}$, ordered by inclusion. It suffices to give a family $\mathcal{B}$ of vertices of $D$ with the property that, for any set $X\subseteq [k-1]$, either $X\in \mathcal{B}$ or $X$ ``points'' to a set in $\mathcal{B}$, i.e., $(X,B)\in A$ for some~$B \in \mathcal{B}$. We construct this family, which we call ``dominating'', using the probabilistic method, similarly as done in \cite[Theorem 1.2.2]{alon2016probabilistic} for bounding the domination number of undirected graphs. 
  
  Let $\cX\subseteq V$ be formed by picking each subset of $[k-1]$ with a probability $p_i$ (to be specified later) that depends on its size $i$. Let $\cY\subseteq V$ contain the sets of $V\setminus \cX$ that do not point to any set in $\cX$. Notice that the family $\cX\cup \cY$ is dominating. We now bound its expected size. We have that
  \begin{align*}
    \mathbb{E}(|\cX|) &=\sum_{i=1}^{k-1} {k-1 \choose i} p_i,\\
    \mathbb{E}(|\cY|)&\leq \sum_{i=1}^{k-1} {k-1 \choose i-1} (1-p_i)^{k-i},
  \end{align*}
  where for the latter we used the fact that a set of size $i-1$ points to exactly $k-i$ sets, none of which has to be picked in $\cX$.
  We now set $p_i=\frac{\ln(k-i)}{k-i}$, and use the fact that $1-p\leq
  e^{-p}$ for all reals $p$. We then get
  \begin{align*}
    \mathbb{E}(|\cX\cup \cY|)
    &\leq \sum_{i=1}^{k-1}\Big( {k-1 \choose i}
    \frac{\ln(k-i)}{k-i}+{k-1 \choose i-1} \frac{1}{k-i}\Big)\\
    &\leq 2\ln(k) \sum_{i=1}^{k-1} \frac{{k \choose i}}{k-i}= 2\ln(k) \sum_{i=1}^{k-1} \frac{{k \choose k-i}}{k-i}\\
    &=2\ln(k)  \sum_{i=1}^{k-1}  \frac{{k+1 \choose k-i+1}}{k+1}\cdot \frac{k-i+1}{k-i}\leq 4\frac{\ln(k)}{k+1} \sum_{i=1}^{k-1}  {k+1 \choose k-i+1}\leq 2^{k+3}\frac{\ln(k)}{k+1},
  \end{align*}
  where the second equality follows from the known identity $\frac{{k \choose \ell}}{\ell+1}=\frac{{k+1 \choose \ell+1}}{k+1}$, and the last inequality from the binomial expansion.
\end{proof}

We can now obtain our bound on $\sucn_h(\{0,1\}^k)$ (hence, on $\slcn_h(\{0,1\}^k)$) by putting together Propositions \ref{prop:boundF1} and~\ref{prop:boundF0}.

\begin{cor}\label{cor:sucn}
  We have $\sucn_h(\{0,1\}^k)\in O(\frac{2^k}{\sqrt{k}})$.
\end{cor}

\subsection{Putting All Pieces Together}

With the preparatory work of the preceding sections, we are able to provide a
proof of Theorem~\ref{thm:main}.

\begin{proof}[Proof of Theorem~\ref{thm:main}]
  First, assume that there is a positive integer~$k$ such that~$d =
  2^{k}-1$.
  Let~$m = d - k$.
  Furthermore, let~$\tilde{\Delta}_d$ as well as~$T = X \cup Y$ be defined
  as in Section~\ref{sec:project}.
  Then, $X = \Delta_k$ and~$Y = \{0,1\}^k \setminus X$.
  Moreover, as defined at the end of Section~\ref{sec:lift}, let~$h\colon \{0,1\}^k
  \to \R$ be given by~$h(x_1, \dots, h_k) = (2x_k - 1) \cdot
  r(x_1,\dots,r_{k-1})$.
    By Lemma~\ref{lem:X:oplus:Delta:simp},
  \[
    \rc(\Delta_d)
    =
    \rc(\Delta_k \oplus \Delta_m)
    \leq
    \rc(X \cup Y)
    +
    \sucn_h (X \cup Y)
    +
    \slcn_h (X \cup Y).
  \]
  As~$X \cup Y = \{0,1\}^k$ and~$\rc(\{0,1\}^k) \leq 2k$, we derive from
  Proposition~\ref{prop:up=low} and Corollary~\ref{cor:sucn} the upper
  bound~$\rc(\Delta_d) \in O(\frac{2^k}{\sqrt{k}})$.
  Because~$k = \log_2(d+1)$, the assertion follows.

  If there is no positive integer~$k$ with~$d = 2^k - 1$, then let~$d'$ be
  the smallest integer greater than~$d$ that admits such a representation.
  Note that~$d < d' \leq 2d$ and~$\rc(\Delta_d) \leq \rc(\Delta_{d'})$.
  Thus, by the above argumentation,
  \[
    \rc(\Delta_d)
    \leq
    \rc(\Delta_{d'})
    \in
    O(\tfrac{d'}{\sqrt{\log(d')}})
    =
    O(\tfrac{d}{\sqrt{\log(d)}}).
  \]
\end{proof}

\section{Outlook} \label{outlook} 

\newcommand{\bF}{{\mathbb{F}}} 
\newcommand{\Ralg}{{\mathbb{R}_{\operatorname{alg}}}}

In light of our new results, we formulate the following open problems: 

\begin{enumerate}[label=(P\arabic*) ]
	\item\emph{For given $d \in \Z_{>0},$ characterize all finite subsets $X$ of $\Z^d$ satisfying \mbox{$\rc(X) < \rc_\Q(X)$.}} The first open case is $d=5$.
	\item \emph{Improve  the asymptotic upper bound on  $\rc(\Delta_d)$.} The best known lower bound on~$\rc(\Delta_d)$ is of asymptotic order $\Omega(\log d)$, see \cite{averkov2021complexity}. In contrast, our asymptotic upper bound $O( \frac{d}{\sqrt{\log d}})$ is only mildly sublinear in $d$. It is widely open what asymptotic behavior for $\rc(\Delta_d)$ one should expect, but it is quite likely that our upper bound is not tight. 
	\item \label{P3} As $\rc(X)$ and $\rc_\Q(X)$ do not coincide in general, there are two algorithmic problems that one can ask, for every fixed dimension $d$ as well as for an unspecified $d$: \emph{Determine whether $\rc(X)$ and $\rc_\Q(X)$ are algorithmically computable.} The existence of sets $X$ satisfying $\rc(X) < \rc_\Q(X)$ is an additional complication that makes an algorithmic computation of these two values quite an intricate task, as explained in~\cite{AverkovEtAl2022}. 
	The authors of \cite{AverkovEtAl2021}
	introduce a non-increasing sequence~$(u_k)_{k \in \N}$ of upper bounds on~$\rc_\Q(X)$, which are computable for each given $k$. This sequence stabilizes by attaining 
	the value~$\rc_\Q(X)$ after finitely many steps.
	However, it is open whether one can decide for fixed~$k$ whether~$u_k =
	\rc_\Q(X)$.
	One way to decide this is by providing a computable matching lower bound.
        In~\cite{kaibel2015lower}, a mechanism to derive lower bounds via so-called hiding sets is presented.
        Deriving a lower bound~$\ell$ via hiding sets yields~$\ell \leq \rc(X) \leq \rc_\Q(X) \leq u_k$.	
        Thus, this approach can only work if~$\rc(X) = \rc_\Q(X)$.
	In~\cite{averkov2021complexity} it is shown that $\rc(X)=\rc_\Q(X)$ for any
	$X\subseteq \Z^d$ when $d\leq 4$, and for other special cases.
        In this case, one might wonder about a computable lower bound on~$\rc(X)$ that is tight.
	Due to Corollary~\ref{cor:dimd} and Theorem~\ref{thm:main}, however,
	$\rc(X)$ can be strictly smaller than~$\rc_\Q(X)$, and thus, matching upper bounds on~$\rc_\Q(X)$ and lower bounds on~$\rc(X)$ do not exist in general.
	\item \emph{Determine the minimum of $\rc(X)$ over all sets $X \subseteq \Z^d$ of dimension $d=\dim(X)$}. This quantity can be analyzed asymptotically, for $d \to \infty$, as well as for concrete choices of $d$. It would be interesting to see if the minimum is attained for $X= \Delta_d$. 
	\item 	\emph{Study the relaxation complexity with respect to arbitrary sub-fields of $\R$.} 
	One can introduce the relaxation complexity $\rc_\bF(X)$ of $X \subseteq \Z^d$ with respect to relaxations of $X$ described by inequalities with coefficients in $\bF$. It is clear that $\rc_\bF(X)$ is monotone in $\bF$ with respect to inclusion. Our arguments show that the estimate $\rc_\bF(X) \in O(\frac{d}{\sqrt{\log(d)}})$  holds for every field of dimension at least $d+1$ over $\Q$, as we only need $d+1$ numbers in $\R$ that are linearly independent over $\Q$ to prescribe appropriate heights. 
	
	For $d=5$, our construction shows $\rc_{\Q[\sqrt{2}]}(\Delta_5)=5$, where the field $\Q[\sqrt{2}]$ has dimension two over $\Q$. This is the smallest dimension of $\bF$ over $\Q$ with $\rc_{\bF}(\Delta_5)=5$, as the only one-dimensional field over $\Q$ is $\Q$ itself. 	 
	It is likely that $\rc_\bF(\Delta_5)=5$ holds for every field $\bF \subset \R$ of dimension two over $\Q$, which just means that $\rc_{\Q[r]}(\Delta_5) = 5$ should be true for every $r \in \R \setminus \Q$. 
	\item \emph{When  $\bF$ is the field of real algebraic numbers, determine if $\rc_\bF(X)$ is computable for a given finite $X \subset \Z^d$ and if the equality $\rc_{\bF}(X) = \rc(X)$ holds for every finite~$X \subset \Z^d$.}
	Indeed, our results show that the choice between rationals and arbitrary reals makes a difference, but the reals in their whole generality include numbers that one cannot even approximate algorithmically, so it makes sense to ask if some reasonably constructive subfield of the reals would be enough for computing the relaxation complexity. The field of algebraic numbers could be a good candidate.
\end{enumerate}

\paragraph*{Acknowledgments.} We would like to thank Matthias Schymura for his valuable comments on this work. We are also grateful to Volker Kaibel and Stefan Weltge.

\bibliographystyle{plainnat}
\bibliography{references} 

\end{document}